\documentclass[11pt, oneside]{amsart}
 \usepackage[text={5.58in,8.5in},centering,letterpaper,dvips]{geometry}
\usepackage{graphicx}
\usepackage{amsfonts}
\usepackage[dvipsnames]{xcolor}
\usepackage{subcaption}
\usepackage{epsf}
\usepackage{amssymb}
\usepackage{amsmath}
\usepackage{amscd}
\usepackage{pdfpages}
\usepackage{fancyhdr}
\usepackage{setspace}
\usepackage{url}
\usepackage{overpic}
\usepackage{soul} 
\usepackage[all]{xy}
\usepackage{verbatim}
\usepackage{enumerate}
\usepackage[colorlinks=true, urlcolor=NavyBlue, linkcolor=NavyBlue, citecolor=NavyBlue]{hyperref}

\theoremstyle{theorem}
\newtheorem{theorem}{Theorem}[section]
\newtheorem{proposition}[theorem]{Proposition}

\newtheorem{conjecture}[theorem]{Conjecture}

\newtheorem*{GPRC}{Generalized Property R Conjecture}
\newtheorem*{SGPRC}{Stable Generalized Property R Conjecture}
\newtheorem*{WGPRC}{Weak Generalized Property R Conjecture}

\makeatletter
\newtheorem*{rep@theorem}{\rep@title}
\newcommand{\newreptheorem}[2]{%
\newenvironment{rep#1}[1]{%
 \def\rep@title{#2 \ref{##1}}%
 \begin{rep@theorem}}%
 {\end{rep@theorem}}}
\makeatother

\newreptheorem{theorem}{Theorem}
\newreptheorem{lemma}{Lemma}
\newreptheorem{question}{Question}
\newreptheorem{corollary}{Corollary}
\newreptheorem{proposition}{Proposition}

\theoremstyle{definition}

\newtheorem{remark}[theorem]{Remark}

\makeatletter
\def\@seccntformat#1{%
  \protect\textup{\protect\@secnumfont
    \ifnum\pdfstrcmp{subsection}{#1}=0 \bfseries\fi
    \csname the#1\endcsname
    \protect\@secnumpunct
  }%
}  
\makeatother

\begin{document}

\rhead{\thepage}
\lhead{\author}
\thispagestyle{empty}


\raggedbottom
\pagenumbering{arabic}
\setcounter{section}{0}


\title[]{Some experimental results on stable equivalence of GST Links for the Generalized Property R Conjecture}

\author{Wenjie Diao}
\address{School of Mathematical Sciences, East China Normal University, Shanghai 200241, China}
\email{740997225@qq.com; wjdiao98@stu.ecnu.edu.cn}

\author{Haoqian Pan}
\address{School of Mathematical Sciences,  Key Laboratory of MEA(Ministry of Education) \& Shanghai Key Laboratory of PMMP, East China Normal University, Shanghai 200241, China}
\email{52215500040@stu.ecnu.edu.cn; hqpan@math.ecnu.edu.cn}

\address{China Communications Information \& Technology Group Co., Ltd., Beijing 101399, China}

\author{Chunxing Yan}
\address{School of Finance and Mathematics, Huainan Normal University, West Dongshan Road, Huainan, Anhui, 232038, China}
\email{cxyan@hnnu.edu.cn}

\begin{abstract}
Gompf-Scharlemann-Thompson and Meier-Zupan constructed an infinite family of R-links that are potential counterexamples of the generalized property R conjecture. Their works also show that whether these links are stably handleslide trivial is an interesting open problem related to the Slice-Ribbon conjecture. In this work, we implement an algorithm to construct all these links explicitly, the details of this algorithm will the content of another paper. With such an algorithm, the stable handleslide triviality of some of these links is verified. Moreover, many links are shown to be stably handleslide equivalent. Some of the results are obtained independently in \cite{Knots in the fiber}.
\end{abstract}

\maketitle
\section{Introduction} \label{chap:intro}
The relationship between Dehn surgery on links in $S^3$ and the topology of 4-manifolds has been a central theme in low-dimensional topology since the inception of Kirby calculus. A fundamental question in this area is how the geometric properties of a link reflect differential structure of the 4-manifolds obtained by surgery.
The Generalized Property R Conjecture (GPRC) stands at the intersection of this relationship, providing a bridge between 3-dimensional surgery theory and 4-dimensional topology.

Recall that Property R Conjecture, famously resolved by Gabai \cite{Gab87} in 1987, asserts that if $0$-framed surgery on a knot $K \subset S^3$ yields $S^2 \times S^1$, then $K$ must be the unknot. Gabai's result can be viewed as initial progress toward a positive resolution of another famous conjecture, the Smooth 4-Dimensional Poincare Conjecture(S4PC), since it implies that any closed, simply connected, smooth 4-manifold which admits a handle decomposition without 1-handle and consisting of  a single 2-handle must be diffeomorphic to the standard $S^4$.

The Generalized Property R Conjecture, problem 1.82 in Kirby's problem list \cite{Kirby problem}, extends Property R Conjecture to links with $n$ components.

\begin{GPRC}
	Let $L = K_1 \cup \dots \cup K_n$ be an $n$-component link in $S^3$. If $0$-framed surgery on $L$ yields the connected sum of $n$ copies of $S^2 \times S^1$, denoted by $\#_n(S^2 \times S^1)$, then $L$ is handle-slide equivalent to the $n$-component unlink.
\end{GPRC}

An $n$-component link $L$ with
the property that $0$-surgery on $L$ yields $\#^n S^1\times S^2$ is called an R-link (or an nR-link
when we emphasize the number of components). Thus Gabai's result establishes that unknot is the only 1R-link

In \cite{GST}, the authors studied GPRC with special focus on 2R links. They concentrated on links $L = U\cup V$ where $U$ is fibered. Using Heegaard theory, they completely characterized  such links $L$: After handleslides, $V$ is a special type of curve lying on the fiber surface of $U$. For the square knot $Q_{3,2}$, which is the simplest plausible fibered knot, such characterization can be further explored such that they enumerated all  2R links $L = Q_{3,2} \cup V$ up to handle slides and all such links can be slid to become of a special form, which is later denoted by  $L(3,2; \frac{c}{d})$ by Meier-Zupan in \cite{Generalized square knot} with $c,d$ coprime integers and $d$ even.  We will describe this construction more explicitly in later section. 

By using an unusual family of handlebodies from \cite{Gompf}, Gompf-Scharlemann-Thompson \cite{GST} constructed a family of 2R links denoted by $L_n$, each of which contains the square knot $Q_{3,2}$ as a constituent knot and proposed that $L_n$  probably
does not satisfy the Generalized Property 2R Conjecture. In view of the characterization of such $L_n$ above, Scharlemann \cite{Proposed} later showed that $L_n$ must be handleslide equivalent to  $L(3,2; \frac{2n}{2n+1})$.

The reason  $L_n$ might be counterexamples to GPRC is that if $L_n$ has Property R, then the
trivial group presentation
$$P_n = \langle x,y\ |\ xyx= yxy,x^n
= y^{n+1}\rangle$$
satisfies the Andrews-Curtis Conjecture, which is widely believed not to be the case when $n\ge3$. See \cite{GST} for further details about the Andrews-Curtis Conjecture.

Thus they also proposed weaker and more plausible versions of the Generalized Property R Conjecture, the Stable Generalized Property R Conjecture(sGPRC) and the Weak Generalized Property R Conjecture(wGPRC). They showed that $L_n$ satisfy wGPRC and whether they satisfy sGPRC and GPRC is an interesting open problem.

\begin{SGPRC}
For every R-link
$L$, there is a 0-framed unlink $U$ such that the split union $L\sqcup U$ is handleslide-equivalent
to an unlink.
\end{SGPRC}

\begin{WGPRC}
For every R-link
$L$, there is a 0-framed unlink $U$ and a split collection of Hopf pairs $V$
such that the split union $L\sqcup U \sqcup V$ is handleslide-equivalent to an unlink and a split
collection of Hopf pairs.
\end{WGPRC}

The family $\{L(3,2;\frac{c}{d})\}$ of 2R links, each of which contains the square knot $Q_{3,2}$ as a constituent knot, was later generalized in \cite{Generalized square knot} into nR links $L(p,q;\frac{c}{d})$ for more general coprime $p,q$ and $\frac{c}{d}\in \mathbb Q$ and $c$ even. Each of these $n$ component link $L(p,q;\frac{c}{d})$ contains the generalized square knot $Q_{p,q}$ as a constituent knot with $n = (p-1)(q-1)$. Most of these links also appear to be
potential counterexamples to the GPRC. They also proved that these more general R links satisfy wGPRC. The question whether these links satisfy GPRC or sGPRC is also an interesting open problem.

Much of the difficulty of understanding these links comes from the complexity of the links themselves: even though each of them is link with a constituent knot the generalized square knot, the other components can be quite complicated and they are interleaved in complicated ways. As such, it is plausible to make use computer algorithms and softwares like snapPy to have a better understanding of these links. 

In this work, we focus on the Stable Generalized Property R Conjecture and describe a computer program that generates all  links of the form $L(p,q; \frac{c}{d})$. Since this link is actually well-defined up to handleslides, what we actually construct are certain 2R and 3R sublinks of a representative link in the handleslide equibalence or stable handleslide equivalence class of $L(p,q; \frac{c}{d})$. This well satisfy our goal to understand the handleslide or stable handleslide equivalence of these links: by making use of such a program, we find many interesting stable equivalence between different parameters of $L(p,q; \frac{c}{d})$.

We first gather what's known on the handle slide equivalence of $L(p,q; \frac{c}{d})$:

\begin{itemize}
\item For the 2R links $L_n$, it is proved in \cite{Proposed} that $L_n$ is handleslide equivalent to $L(3,2; \frac{2n}{2n+1})$. 
\item It is known that $L_n$ satisfy GPRC for $n = 0, 1, 2$ and is not known to satisfy GPRC for $n\ge3$; see \cite{Characterizing}.
\item For the R links $L(2k+1,2;\frac{2}{2m+1})$ with $k\in \mathbb N$ and $m\in \mathbb Z$, it is proved in \cite{Knots in the fiber} that they are handleslide trivial and hence satisfy GPRC.
\item 
For the 2R links $L(3,2;\frac{4}{d})$ with $d\in \mathbb Z$, it is proved in \cite{Romary Zupan} that the associated trivial group presentation is Andrew-Curtis trivial, which indicates that they might also satisfy sGPRC. Actually it is conjectured in \cite{Romary Zupan} that they satisfy sGPRC.
\end{itemize}

Our main results are the following:

\begin{theorem}
\

\begin{enumerate}
\item The GST link $L_{n}$ is stably handleslide equivalent to $L(n+1,n;\frac{2}{3})$ for $2\le n\le 36$.
\item The GST link $L(3,2;\frac{4}{d})$ is stably handleslide trivial for $1\le d \le 39$.
\item The GST link $L(3,2;\frac{6}{6n\pm 1})$ is stably handleslide equivalent to $L(3n\pm 1,3;\frac{2}{3})$ for $n \le 7$.
\end{enumerate}
\end{theorem}

\begin{remark}
\begin{enumerate}
    \item The stable triviality of the link $L(3,2;\frac{4}{d})$ for $1\le d\le 23$ is independently obtain in \cite{Knots in the fiber}.
    \item We find two different ways to verify the stable handleslide triviality of $L(3,2; \frac{4}{d})$: one by verifying that there are a pair of isotopic 2-component links in $L(3,2; \frac{4}{d})$ and $L(3,2;\frac{4}{d+2})$. By showing that $L(3,2;\frac{4}{5})$ is stably handleslide trivial, this shows that each of $L(3,2;\frac{4}{d})$, for $d\le 39$ is stably handleslide trivial. The other way is to verify that $L(3,2;\frac{4}{4n\pm 1})$ is stably handleslide equivalent to $L(2n\pm 1,2 ;\frac{2}{3})$. Since $L(2n\pm 1,2 ;\frac{2}{3})$ is handleslide trivial by \cite{Knots in the fiber}( our experiments also verify this for small $n$), it also follows that $L(3,2;\frac{4}{d})$ is stably handleslide trivial for $d\le 39$. 
\end{enumerate}
\end{remark}

Motivated by the above results, we formulate the following conjectures:
\begin{conjecture}

\begin{enumerate}
    \item The GST link $L(3,2;\frac{6}{6n \pm 1})$ is stably handleslide equivalent to $L(3n\pm 1,3;\frac{2}{3})$ for $n \ge 1$ where we require that $6n \pm1 >6$ and $3n \pm 1 >2$. 
\item The GST link $L(3,2;\frac{4}{4n\pm 1})$ is stably handleslide equivalent to $L(2n\pm 1,2 ;\frac{2}{3})$ for $n \ge 1$ where we require that $4n \pm1 >4$ and $2n \pm 1 >2$.
\end{enumerate}

\end{conjecture}

Since each link $L(3,2; \frac{4}{d})$ for $5\le d\le 39$ is stably handleslide trivial, each knot that shows up in the handleslide equivalence class of  these links is a  ribbon knot. Building on our general algorithm, we are able to incorporate more general handle slide operations programatically to the link $L(p,q;\frac{c}{d})$. This creates many new knot types and new link types in $L(p,q;\frac{c}{d})$. In particular, every new knot type obtained for the class $L(3,2;\frac{4}{d})$ for $q\le d\le 39$ is automatically a ribbon knot by our result. Given the recent work of \cite{Dunfield-Gong} and \cite{Trevor}, we are interested to see if there is some slice knot which is not known to be ribbon showing up in $L(3,2; \frac{4}{d})$ for $1\le d \le 39$.

\subsection*{Acknowledgements}

This problem was one of a series of problems discussed while the first author was visiting Prof. Alexander Zupan, where the first author learned a lot about GPRC, recent progress on it and all the background materials in this work. This work started after the visiting program, but many ideas originated from the discussion during the visiting, for which we express our gratitude to Prof. Alexander Zupan.

\section{Summary of the algorithm} \label{chap:alg}

\subsection{Descrition of the link}
The generalized square knot $Q_{p,q}$ is defined to be $Q_{p,q} = T_{p,q}\# T_{-p,q}$, where $T_{p,q}$ is the $(p,q)$-torus knot with $0 < q < p$. Let $F^{\pm}$ denote the minimal genus Seifert surface for $T_{p,q}$ and $T_{-p,q}$ respectively. Then both $T_{\pm p,q}$ and the generalized square knot $Q_{p,q}$ are fibered knots, with fiber surface $F^{\pm}$ and $F_{p,q} = F^{+}\natural F^{-}$, where $\natural$ denotes the boundary-connected summation. See Figure \ref{fig:Q_4,3} for the case of $Q_{4,3}$. Moreover, the monodromy map $\varphi^{\pm}$ and $\varphi_{p,q}$ of $F^{\pm}$ and $F_{p,q}$ are periodic of order $pq$.

\begin{figure}[h!]  
\centering
\includegraphics[width=0.8\linewidth]{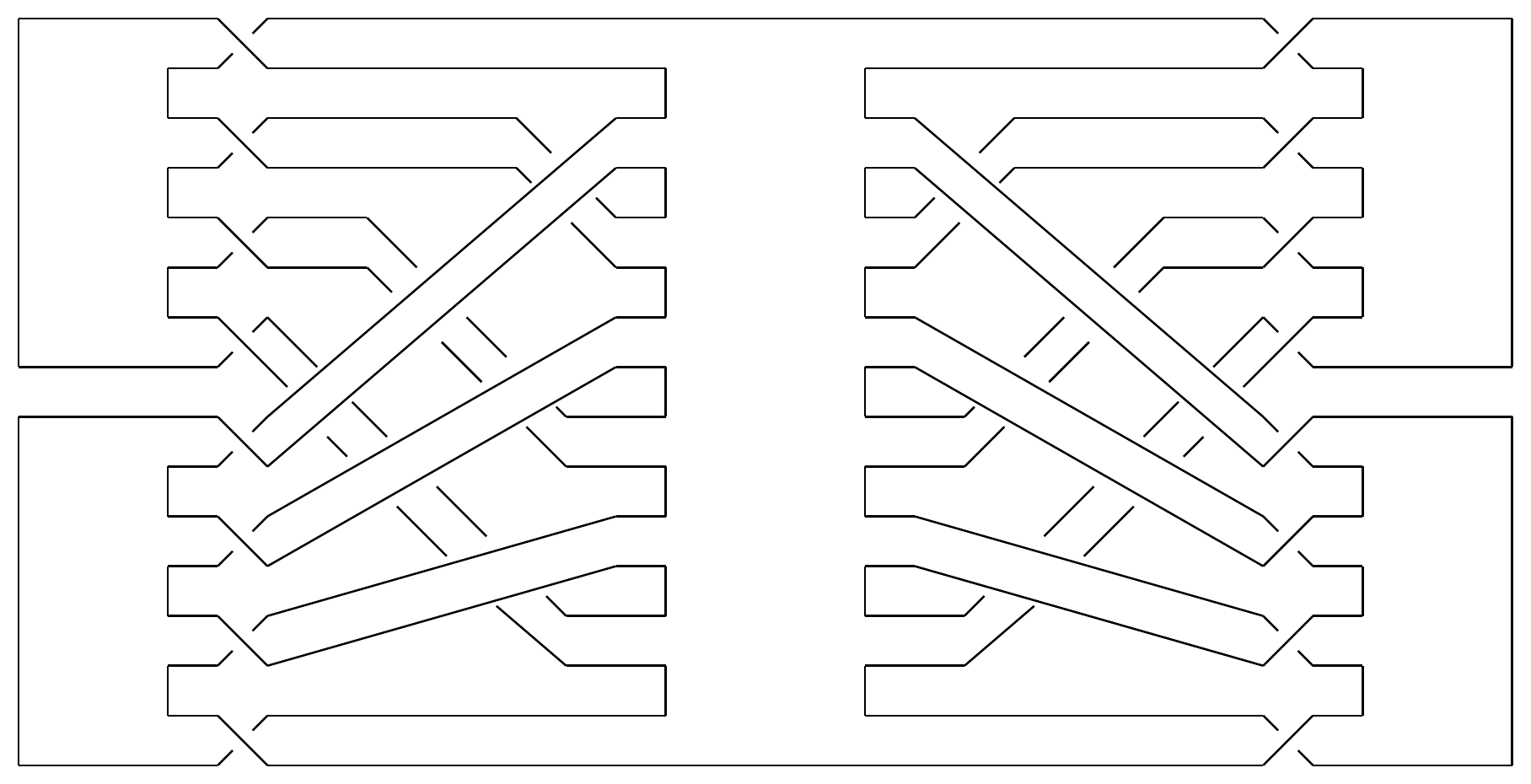}
\caption{The knot $Q_{4,3}$}
\label{fig:Q_4,3}  
\end{figure}

It is explained in \cite{Generalized square knot, GST, Proposed} how to describe this monodromy map $\varphi_{p,q}$, which we describe briefly for our algorithm.

Since $T_{p,q}$ and $T_{-p,q}$ differ by a mirror reflection, it suffices to understand the action of $\varphi^{+}$ on $T_{p,q}$. There is a graph $\Gamma^+$ embedded in the fiber surface $F^+$ which is invariant under $\varphi^+$. By cutting $F^+$ open along the graph $\Gamma^+$, we get an annulus $A^+$, with one boundary component $T_{p,q}$ and the other component a $2pq$-gon coming from $\Gamma^+$. Moreover, the monodromy map $\varphi^+$ acts on the annulus $A^+$ by a $\frac{2\pi}{pq}$ clockwise rotation. Thus, in terms of the number of edges, $\varphi$ is a 2-edge clockwise rotation. For the action of $\varphi^-$ on $F^-$, the situation is similar: cutting $F^-$ along $\Gamma^-$ also gives an annulus $A^-$, with one boundary component $T_{-p,q}$ and the other component a $2pq$-gon coming from $\Gamma^-$. Then $\varphi^-$ acts on $A^-$ by a $\frac{2\pi}{pq}$ clockwise rotation. Gluing $A^+$ and $A^-$ along an arc in $T_{p,q}$ and $T_{-p,q}$ we get a once-punctured annulus $A_{p,q}$ and $\varphi^{\pm}$ extends to $\varphi_{p,q}$, which acts on $A_{p,q}$ also as a $\frac{2\pi}{pq}$ clockwise rotation, with only the difference that $\varphi_{p,q}$ now moves the puncture. However, for our problem, this moving puncture only affects the final realization of the links by handle slides over $Q_{p,q}$, which does not affect much since we are only interested in handleslide and stable handleslide equivalent class. For the following description, we ignore the puncture, so $A_{p,q}$ is just an annulus. 

The quotient space of $A_{p,q}$ by the periodic map $\varphi_{p,q}$ is the so-called pillow case, which is an orbifold $\mathcal O = \mathcal{O}_{p,p,q,q}$ with four cone points of order $p,p,q,q$. The $\frac{c}{d}$ curve in $\mathcal{O}$ is defined as the unique simple closed curve in $\mathcal O$ that avoids all the cone points and winds around horizontally $q$ times and vertically $p$ times, where the orientation is specified by the standard right hand rule.

Then we can lift each of the $\frac{c}{d}$ curve, with $c$ even in $\mathcal O$ into $F_{p,q}$ to get a collection $\mathcal{L}_{p,q}$ of $pq$ disjoint simple closed curves in $A_{p,q}$, depending on the base point of each lift. The first step of the algorithm is to realize this process.

The next step, which is the core part of the algorithm, is to realize the collection of curves $\mathcal{L}_{p,q}$ in $A_{p,q}$ in the Seifert surface $F_{p,q}$ in $S^3$, which we also denote by $\mathcal{L}_{p,q}$. As in the discussion of \cite{Generalized square knot, GST, Proposed}, this realization depends on a process of sliding both ends of one of finitely many model properly embedded arcs in  $F_{p,q}$. Essentially, the $\frac{c}{d}$ parameter decides the number of copies in each model arc. In this step, we first characterize the result of the slide for each model arc. Taking input from the first step, the second step then realizes the actual link in $F_{p,q}$.

In the fiber surface $F_{p,q}$ of genus $(p-1)(q-1)$, every $(p-1)(q-1)$ component sublink of $\mathcal{L}_{p,q}$ that cuts  $F$ into a connected planar surface will be our primary concern. In which case we denote by $L_{\frac{c}{d}}^{p,q} = L$.
Together with $Q_{p,q}$, we get a $(p-1)(q-1)+1$ component link $L(p,q; \frac{c}{d})$ in $S^3$. Since we are mainly interested in stable equivalence, we also denote by $L(p,q;\frac{c}{d})$ any nR link which is stably equivalent to $L(p,q;\frac{c}{d})$ with $n\ge 2$.

Our algorithm explicitly follows the above description to build $L(p,q;\frac{c}{d})$, which can be 2-component or 3-component, whenever we feel it is easier to verify stable equivalence. 
The algorithm itself turns out to be quite technical, so in this paper we only describe what it does. We also believe that some techniques developed in the algorithm might be useful to other problems in knot theory, so we will put the content of the algorithm in another work.

\subsection{$\frac{c}{d}$ path in $A_{p,q}$}

As discussed above, $A_{p,q}$ is an annulus, with each one of the two boundary circles $2pq$-gon, see Figure \ref{fig:cdcircle}. Thus $A_{p,q}$ can be thought of as a regular $2pq$-gon annulus, consisting of $2pq$ disjoint rectangles glued consecutively, in Figure \ref{fig:cdcircle} the rectangles are glued along blue arcs. 

The goal of this step is to lift $\frac{c}{d}$ curve in $\mathcal{O}$ to $A_{p,q}$. There are $pq$ lifts in $A_{p,q}$, depending on the chosen base points. We will label the edges of $2pq$-gon so that this computation can be carried out via arithmetic computation. For example when $p=4$ and $q = 3$, the outer boundary is labeled as in Figure \ref{fig:cdcircle}. 

Each label is a tuple $((a,b,s),t)$ with $1\le a \le p$, $1\le b \le q$, $a,b \in \mathbb{Z}^+ $, $s = + \text{ or } -$ and $t\in \mathbb{Z}_2$. Intuitively, the binary parameter $t$  indicates whether this edge belongs to the outer cicle boundary or not: if the edge belongs to the outer ciecle boundary circle then $t = 0$ and otherwise $t = 1$. Many times the edge with label $((a,b,s),t)$ is denoted for short by $ab_s$ or $(a,b,s)$, since whether the edge belongs to the outer circle boundary is intuitively clear. 

The other parameter $s$ of the label are defined inductively as follows: 

(1)Assigning an arbitrary edge in outer boundary circle $(1,1,+)$; 

(2) For any edge labelled $(a,b,s)$, unless the next edge( counted clockwise) has been labelled already, label it by $(a, b+1 (\text{mod}{q}), s)$, if $s = +$; and lebel the next edge $(a+1 (\text{mod}{p}),b, s)$, if $s = -$. 

(3) When the process (2) terminates, i.e., each edge of the outer circle has been labeled, copy the label $(a,b,s)$ of each edge in outer circle to the corresponding edges in the inner boundary circle, where by the corresponding edge we mean the edge in the inner circle that lies in the same rectangle, with only the $t$ parameter changed from $0$ to $1$.

Thus we labeled the $2pq$-gon annulus $A_{p,q}$. Note that the input $\frac{c}{d}$ has not come into the play yet. 

(4) For any $\frac{c}{d}$, we further introduce an auxiliary parameter $v$ for future use and $v$ takes values in $\{ 1,2, \dots, d\}$. 
We divide each edge of both boundary circles of $A_{p,q}$ into $d+1$  segments of same lenghth by inserting $d$ vertices labeled $1,2,\cdots, d$ consecutively in clockwise order. We denote each of the vertices introduced in this step $((a,b,s),t,v)$ or $v$ when the edge $(a,b,s)$ in which vertex $v$ lies and since whether or not the edge belongs to the outer boundary is intuitively clear , we have omitted $v$ parameter. See Figure \ref{fig:example computation}, where some $v$ parameters are labeled.  

(5) For any edge $((a,b,s),t)$, the next edge of $((a,b,s),t)$, counted clockwise, is denoted by $N(((a,b,s),t))$. If $s = +$, then $N(((a,b,s),t)) = ((a, b+1 (\text{mod} q), -s), t)$. If $s = -$, then $N(((a,b,s),t)) = ((a+1 (\text{mod} p), b, -s), t)$.

\begin{figure}[h]  
\centering
\includegraphics[width=0.6\linewidth]{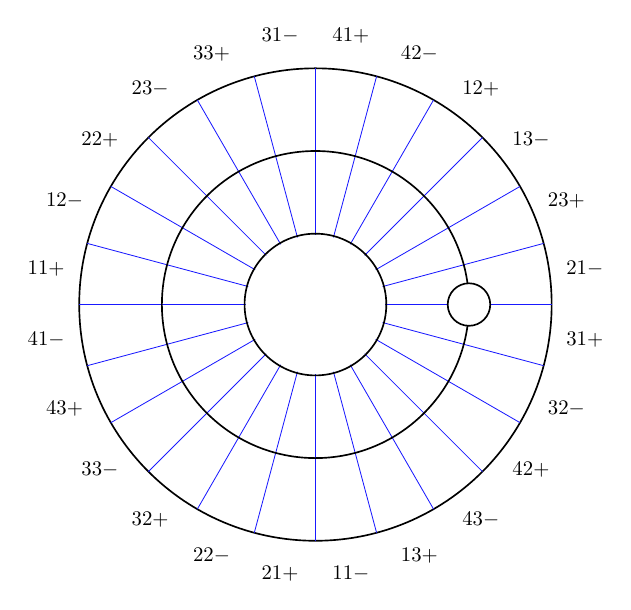}
\caption{$\frac{c}{d}$ circle of $p = 4$ and $q = 3$}
\label{fig:cdcircle}  
\end{figure}

The steps (1)(2)(3) and (4) finishes the labeling process. Now we describe two operations involved in the computation of the lift of 
$\frac{c}{d}$ curve to $A_{p,q}$: translation operator $T$ and reflection operator $R$.

\begin{figure}[h!]  
\centering
\begin{overpic}[width=0.4\linewidth]{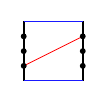}
    \put(15,  33){\color{black}$1$}  
    \put(15,  45){\color{black}$2$}   
    \put(15,  57){\color{black}$3$}   
    \put(82,  33){\color{black}$1$}  
    \put(82,  45){\color{black}$2$}   
    \put(82,  57){\color{black}$3$} 
\end{overpic}
\caption{Translation in the case $\frac{c}{d} = \frac{2}{3}$}
\label{fig:translation case1}  
\end{figure}

\begin{figure}[h!]  
\centering
\begin{overpic}[width=0.4\linewidth]{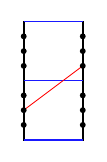}

     \put(10,  20){\color{black}$1$}  
    \put(10,  30){\color{black}$2$}   
    \put(10,  40){\color{black}$3$}   
    \put(55,  20){\color{black}$1$}  
    \put(55,  30){\color{black}$2$}   
    \put(55,  40){\color{black}$3$} 
    \put(10,  20){\color{black}$1$}  
    \put(10,  30){\color{black}$2$}   
    \put(10,  40){\color{black}$3$}   
    \put(55,  58){\color{black}$1$}  
    \put(55,  68){\color{black}$2$}   
    \put(55,  78){\color{black}$3$} 
      \put(10,  58){\color{black}$1$}  
    \put(10,  68){\color{black}$2$}   
    \put(10,  78){\color{black}$3$} 
\end{overpic}
\caption{Translation in the case $\frac{c}{d} = \frac{2}{3}$}
\label{fig:translation case2}  
\end{figure}
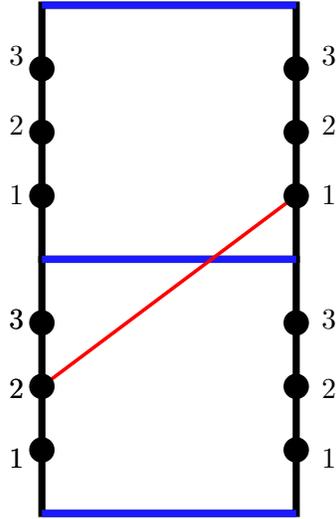

For any vertex $((a,b,s),t,v)$ the translation operator $T$ is defined differently depending on $v$ and $\frac{c}{d}$ as follows: 
\begin{enumerate}
    \item If $v+c <d$, then $T(((a,b,s),t,v)) = ((a,b,s),t+1, v+c)$. An example with $\frac{c}{d} = \frac{2}{3}$ is Figure \ref{fig:translation case1};
    \item If $v+c <d$, let $\Delta = \lfloor \frac{v+c}{d}\rfloor$ and $\delta = v+c -\Delta d$. Then $T(((a,b,s),t,v)) = (N^{\Delta}(a,b,s),t+1, v+c)$. An example with $\frac{c}{d} = \frac{2}{3}$ is Figure \ref{fig:translation case2}
\end{enumerate}

For any vertex $((a,b,s),t,v)$ the reflection operator $R$ is defined by $T(((a,b,s),t,v)) = ((a,b,-s),t, d+1-v)$.

Then the result of $\frac{c}{d}$ computation is the sequence of the vertices obtained in the following steps:
\begin{itemize}
    \item step 1. Choose an arbitrary edge in the outer circle, for example $((1,1,+),0)$, and an arbitraty vertex within that edge, for example $((1,1,+),0,1)$.
    \item step 2. Apply the translation operation to the vertex in step 1 to get another vertex. 
    \item step 3. Apply the reflection operation to the vertex obtained in step 2 to get a new vertex.
    \item step 4. Apply the inverse of the translation operation to the vertex in step 3 to get a new vertex. 
     \item step 5. Apply the reflection operation to the vertex obtained in step 4 to get a new vertex.
     \item step 6. Iterate steps 2-5 until we get the vertex we start with in step 1.
     \item step 7. Record each vertex we get in the above process.
\end{itemize}

We include Figure \ref{fig:example computation} to provide intuition for algebraic computation above: essentially the algebraic computation aims to record the geometric configuration of the sequence of red arcs in Figure \ref{fig:example computation} by recording the vertices of each red arc. In this case, $\frac{c}{d} = \frac{2}{3}$, $p=4, q =3$, and the starting vertex is  $((1,1,+),0,1)$. 

\begin{figure}[h!] 
\centering
\includegraphics[width=0.8\linewidth]{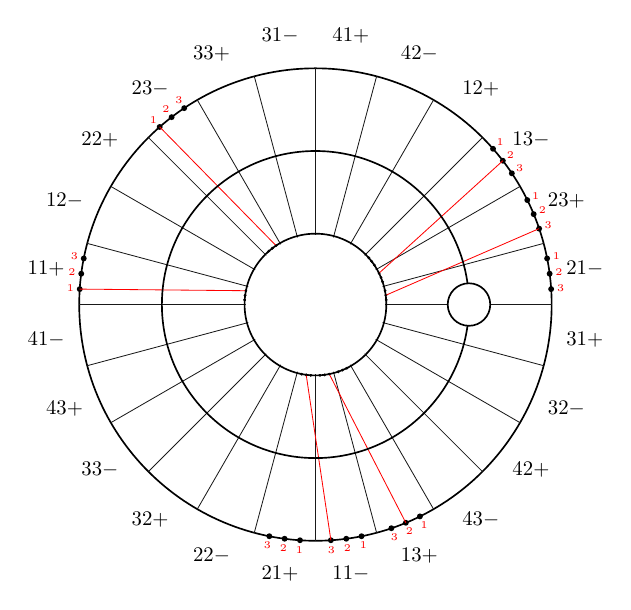}
\caption{An $\frac{c}{d} = \frac{2}{3}$ computation starting at $((1,1,+),0,1)$.}
\label{fig:example computation}  
\end{figure}

\begin{figure}[h!] 
\centering
\includegraphics[width=0.8\linewidth]{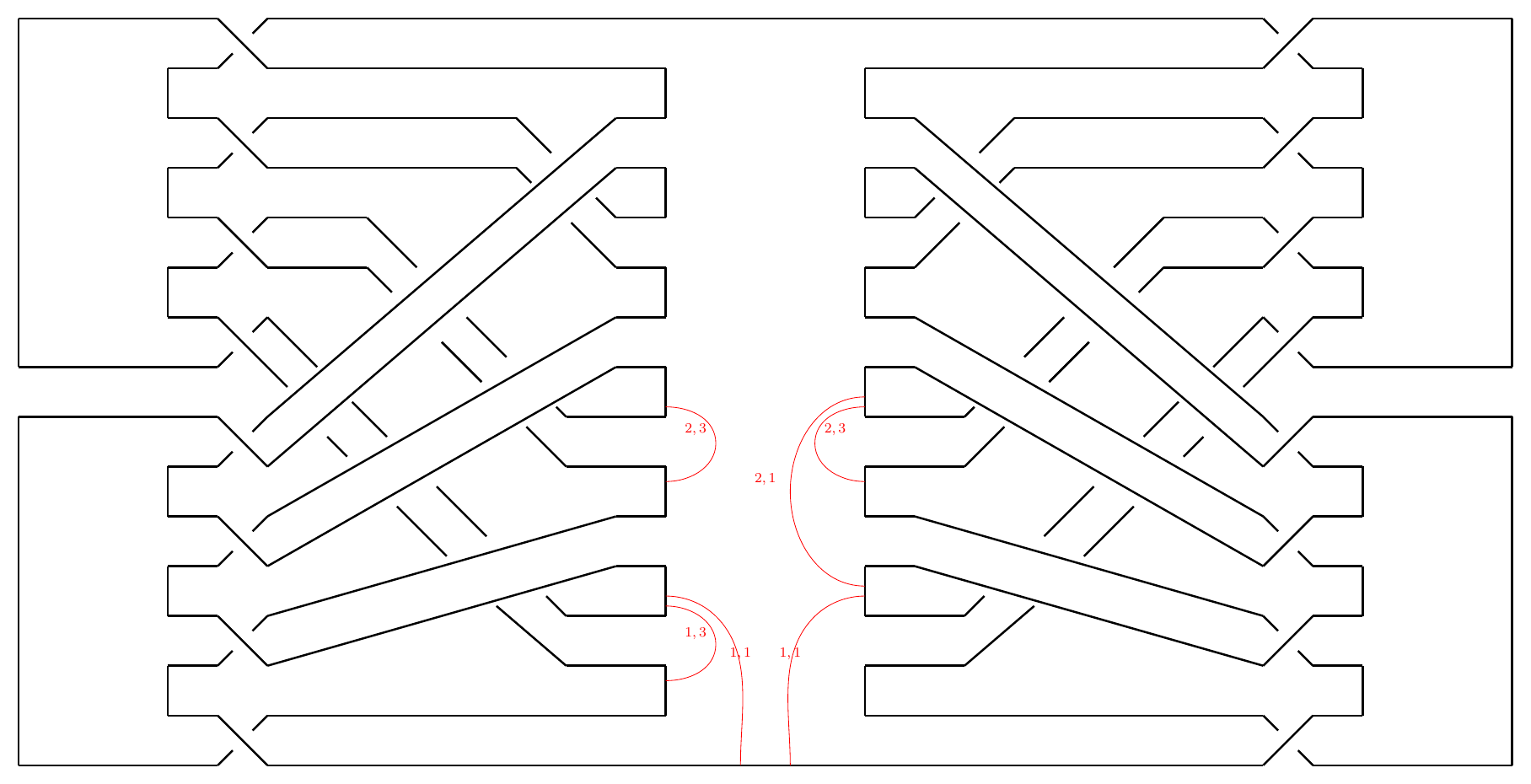}
\caption{The corresponding arcs of Figure \ref{fig:example computation} before the sliding process.}
\label{fig:example model arcs}  
\end{figure}

\begin{figure}[h!]
\centering
\includegraphics[width=0.8\linewidth]{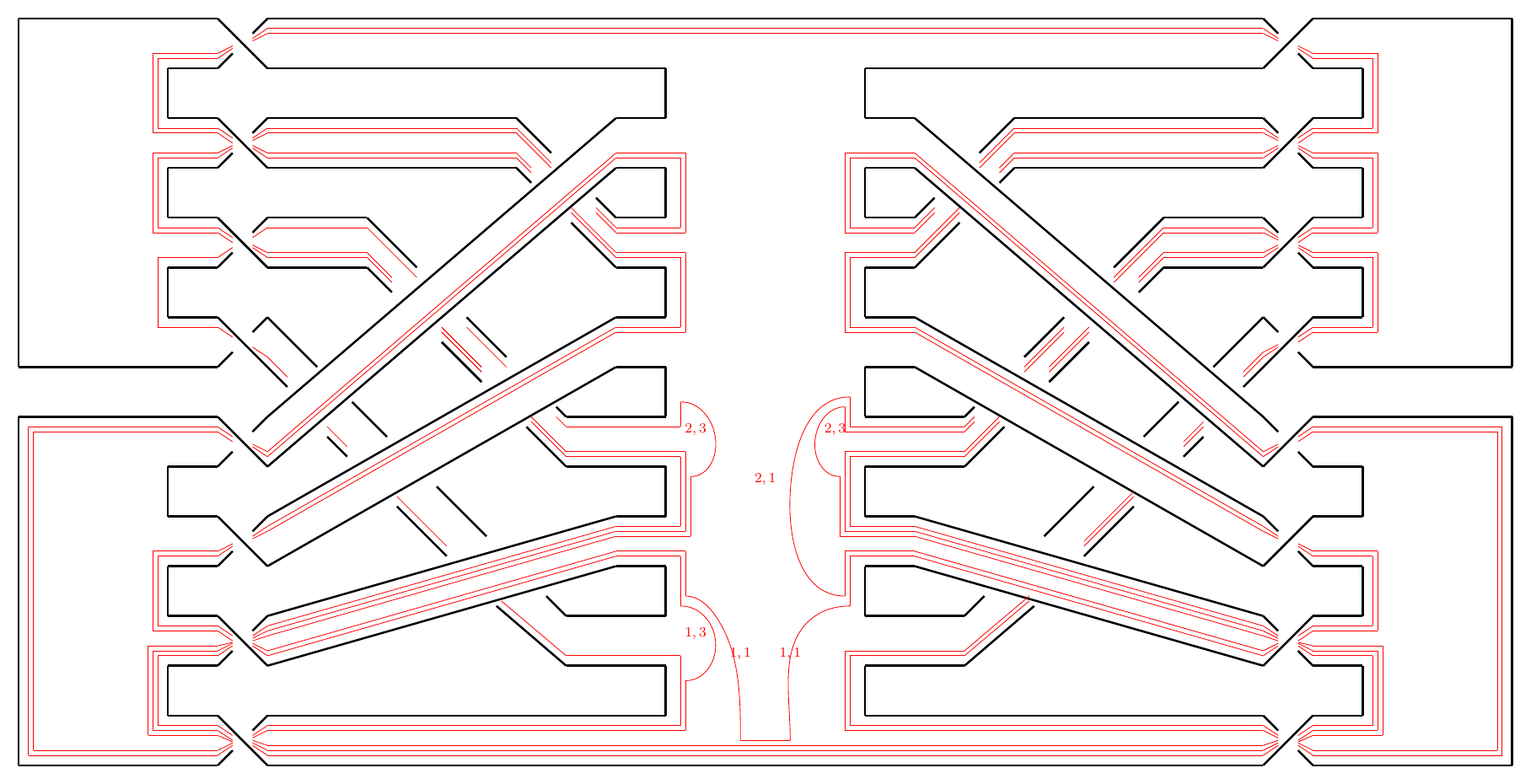}
\caption{The corresponding knot of Figure \ref{fig:example computation} obtained after the sliding process.}
\label{fig:example model arcs after slide}  
\end{figure}

\subsection{Sliding process}

\begin{figure}[h!]
\centering
\includegraphics[width=0.6\linewidth]{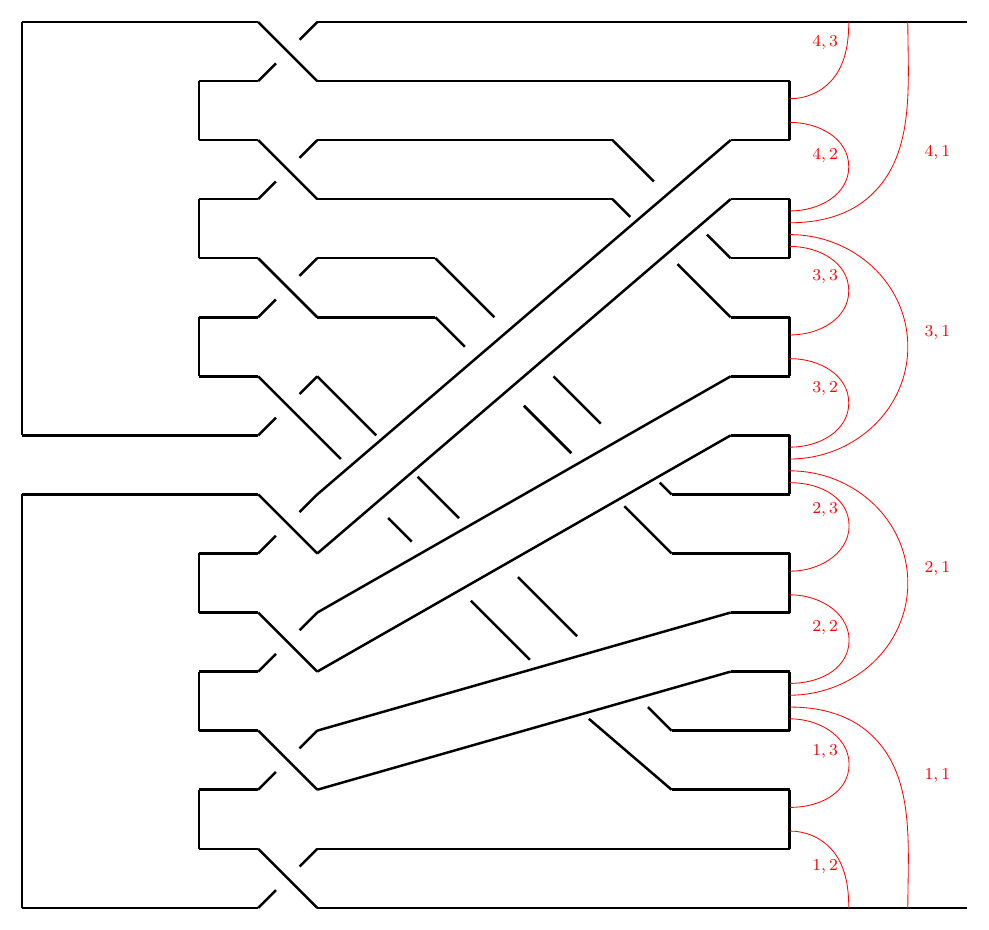}
\caption{For $p=4$ and $q = 3$, there are 12 model arcs in $F^+$.}
\label{fig:43normalgraph}  
\end{figure}

\begin{figure}[h!]
\centering
\includegraphics[width=0.6\linewidth]{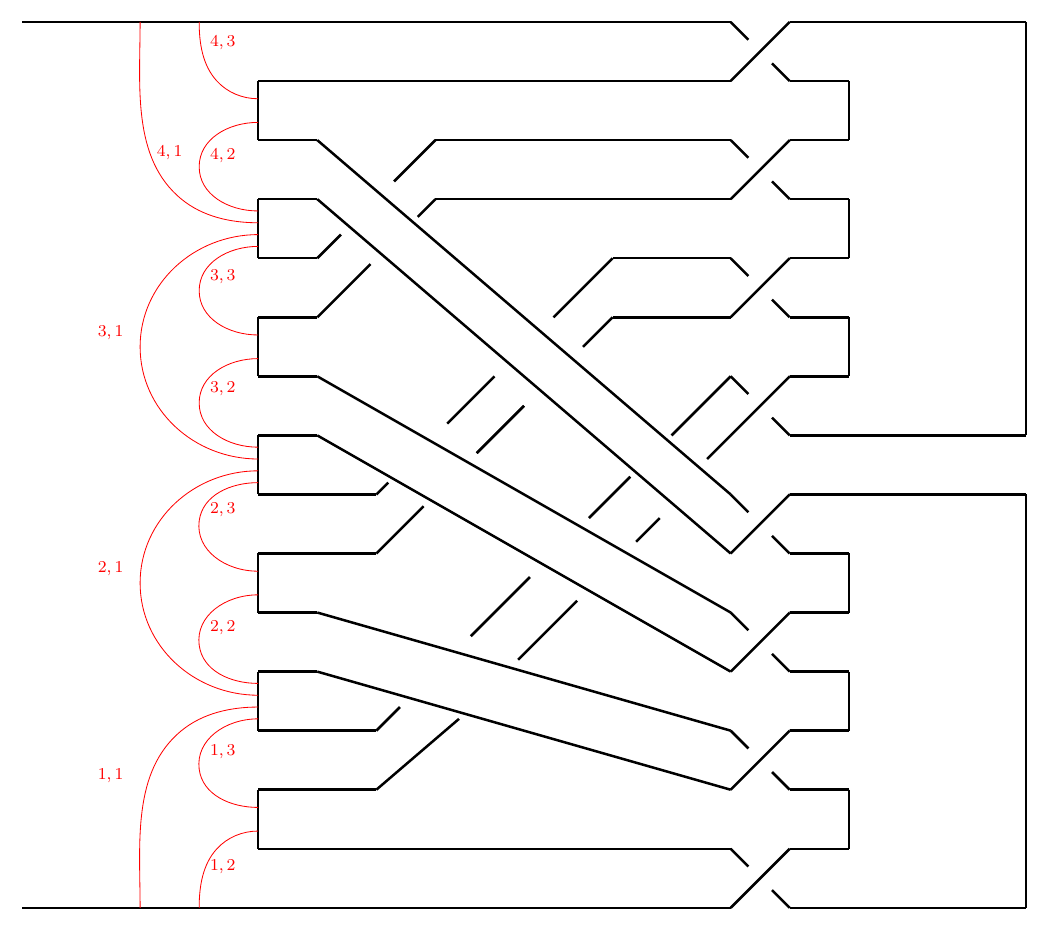}
\caption{For $p=-4$ and $q = 3$, there are also 12 model arcs in $F^-$.}
\label{fig:-43normalgraph}  
\end{figure}

\begin{figure}[h!]
\centering
\includegraphics[width=0.6\linewidth]{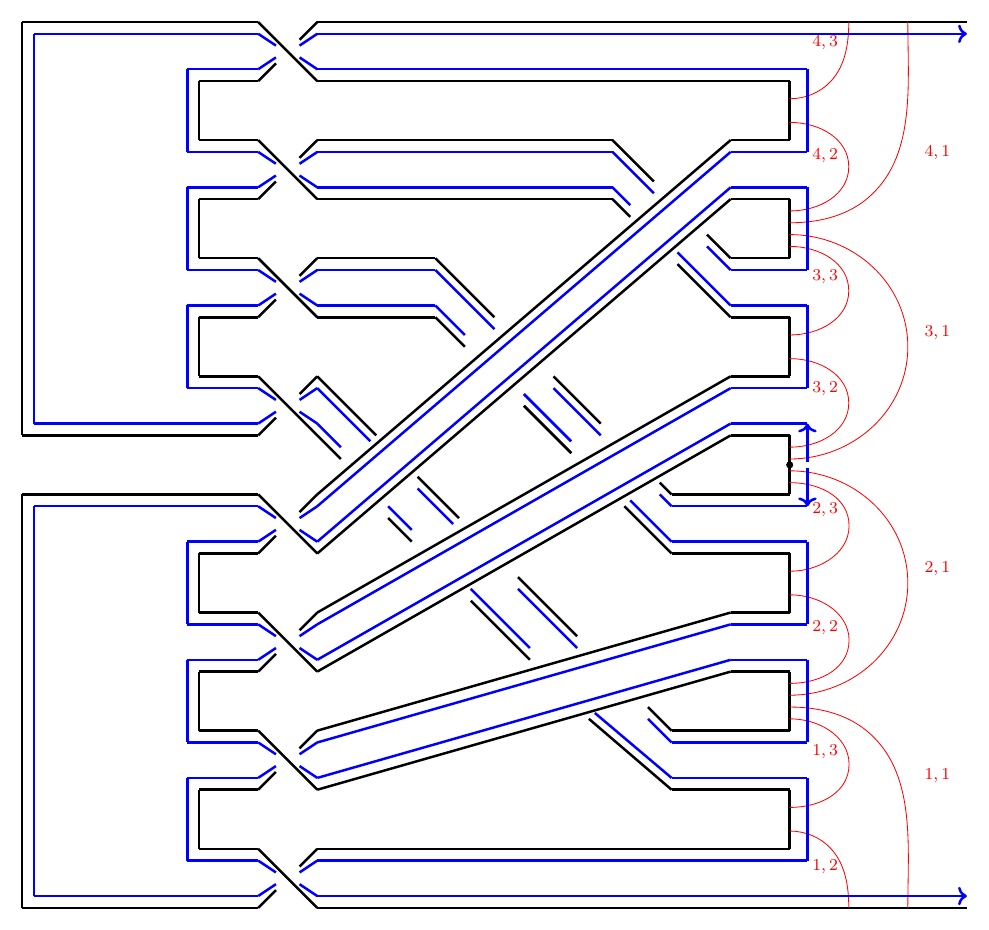}
\caption{For $p=4$ and $q = 3$, blue curves indicate how to slide the ends of each red arc}
\label{fig:43 slide}  
\end{figure}

\begin{figure}[h!]  
\centering
\includegraphics[width=0.6\linewidth]{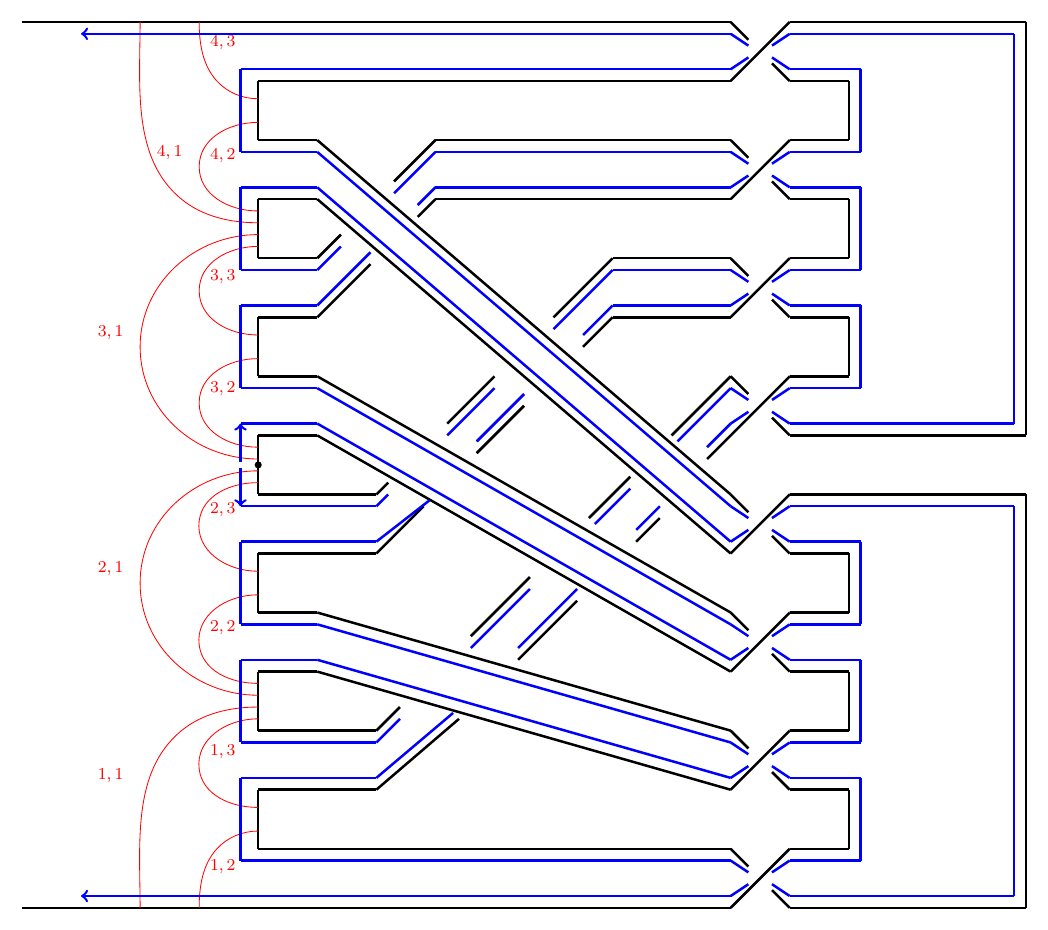}
\caption{For $p=-4$ and $q = 3$, blue curves indicate how to slide the ends of each red arc}
\label{fig:-43 slide}  
\end{figure}

\begin{figure}[h!]
\centering
\includegraphics[width=0.6\linewidth]{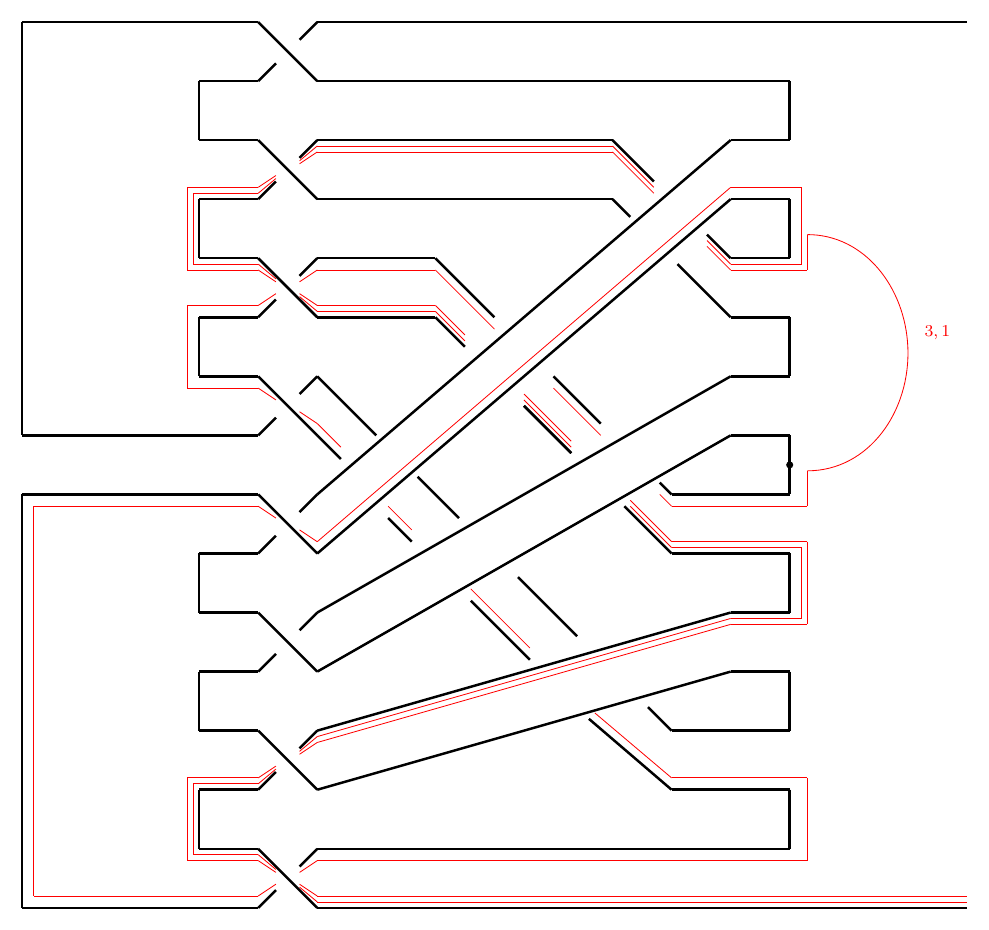}
\caption{For $p=4$ and $q = 3$, the result of sliding $3,1$ arc}
\label{fig:arc slide}  
\end{figure}

The previous step takes $p,q,\frac{c}{d}$ as input, and the number of each model arcs as output. Each of the  model arc labeled $a,b$ in Figure \ref{fig:43normalgraph} corresponds to  a pair of vertices of the form $((a,b,+),0,v)$ and $((a,b,-),0,d-v)$ for some $v$. Similarly in Figure \ref{fig:-43normalgraph}, each of the model arc labeled $a,b$ in its mirror surface $F^-$ corresponds to a pair of vertices of the form $((a,b,+),1,v)$ and $((a,b,-),1,d-v)$ for some $v$. 
To construct an actual link in $F_{p,q}$, we follow \cite{Generalized square knot, GST, Proposed} to slide both ends of each arc along the knot $Q_{p,q}$ so that the ends located near the boundary-connected sum arc in $F_{p,q}$. The sliding process is indicated in  Figure \ref{fig:43 slide} for the case $p=4, q=3$: the  red arcs are the model arcs in $F^+$ and a point on $T_{4,3}$ is chosen so that every end points above and below this point are pushed away along $T_{4,3}$ untill the end points are piled up like the end part of the blue arrowed arcs. The sliding process in the mirror surface $F^-$ is then similarly defined, see Figure \ref{fig:-43 slide}. For example, the result of sliding $3,1$ arc in $F^+$ is illustrated in Figure \ref{fig:arc slide}.

The number of copies of each model arc is essentially the output of the first step. Then applying the sliding process described above creates a pretty complicated collection of arcs in $F^+$, the two ends of each arc are then piled up like the two ends of the blue arrows in Figure \ref{fig:43 slide}. See also Figure \ref{fig:ends}, where we only draw the ends of the result of the slide.

\begin{figure}[h!]
\centering
\includegraphics[width=0.6\linewidth]{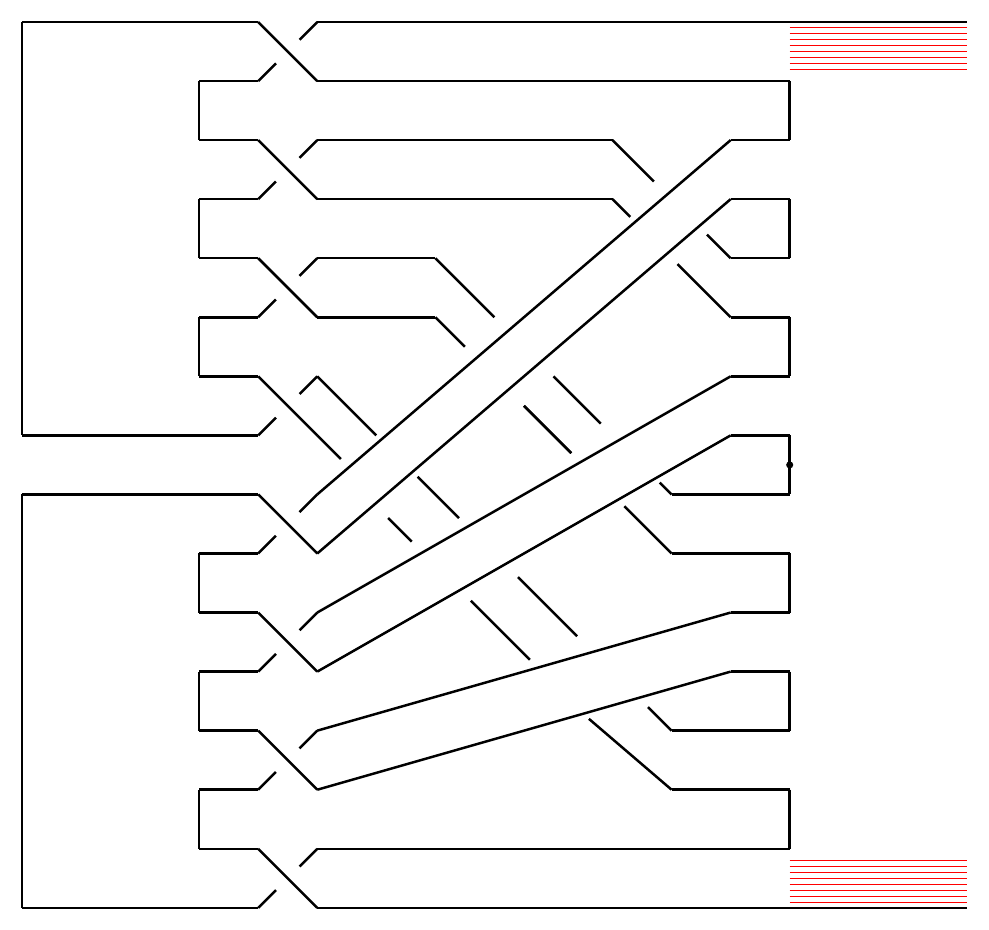}
\caption{For $p=4$ and $q = 3$, the ends of the slide arcs}
\label{fig:ends}  
\end{figure}

Thus we have described the sliding result of in the surface $F^+$.  The sliding process  on the mirrored surface $F^-$ is essentially the mirror of the same process in $F^+$. So the result is also a pretty complicated collection of arcs in $F^-$  with two ends of each arc piled up similar to that of the result of sliding arcs in $F^+$, see Figure \ref{fig: mirrorends}.

By construction in \cite{Generalized square knot, GST, Proposed}, connecting the ends of the red arcs in the natural way provides a single knot which we denoted by $V_{a,b,s,t,v}$ to indicate the is $((a,b,s),t,v)$. The desired 2-component link is $Q_{p,q} \cup V_{a,b,s,t,v}$ , see Figure \ref{fig: connecting ends}. For a concrete example, see Figure \ref{fig:example model arcs} and Figure \ref{fig:example model arcs after slide}.

\begin{figure}[h!]
\centering
\includegraphics[width=0.6\linewidth]{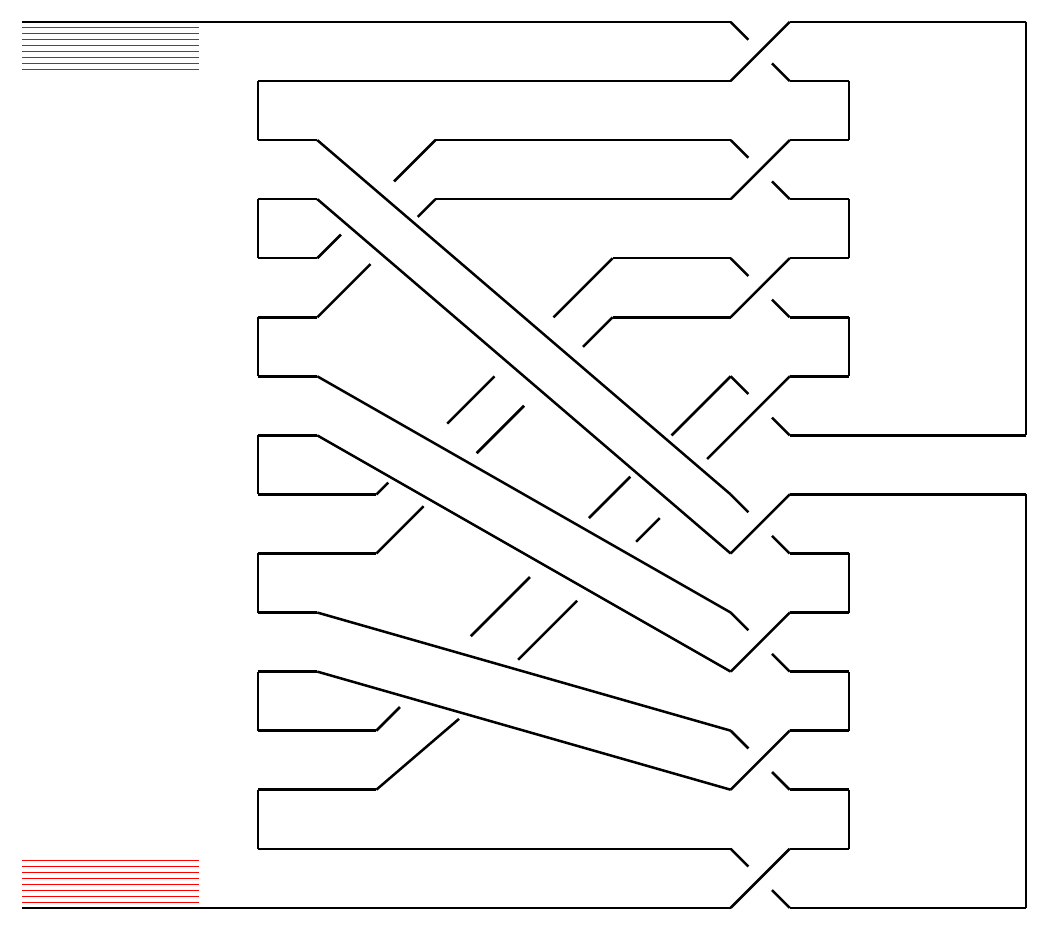}
\caption{For $p=-4$ and $q = 3$, the ends of the slide arcs}
\label{fig: mirrorends}  
\end{figure}
\begin{figure}[h!]  
\centering
\includegraphics[width=0.8\linewidth]{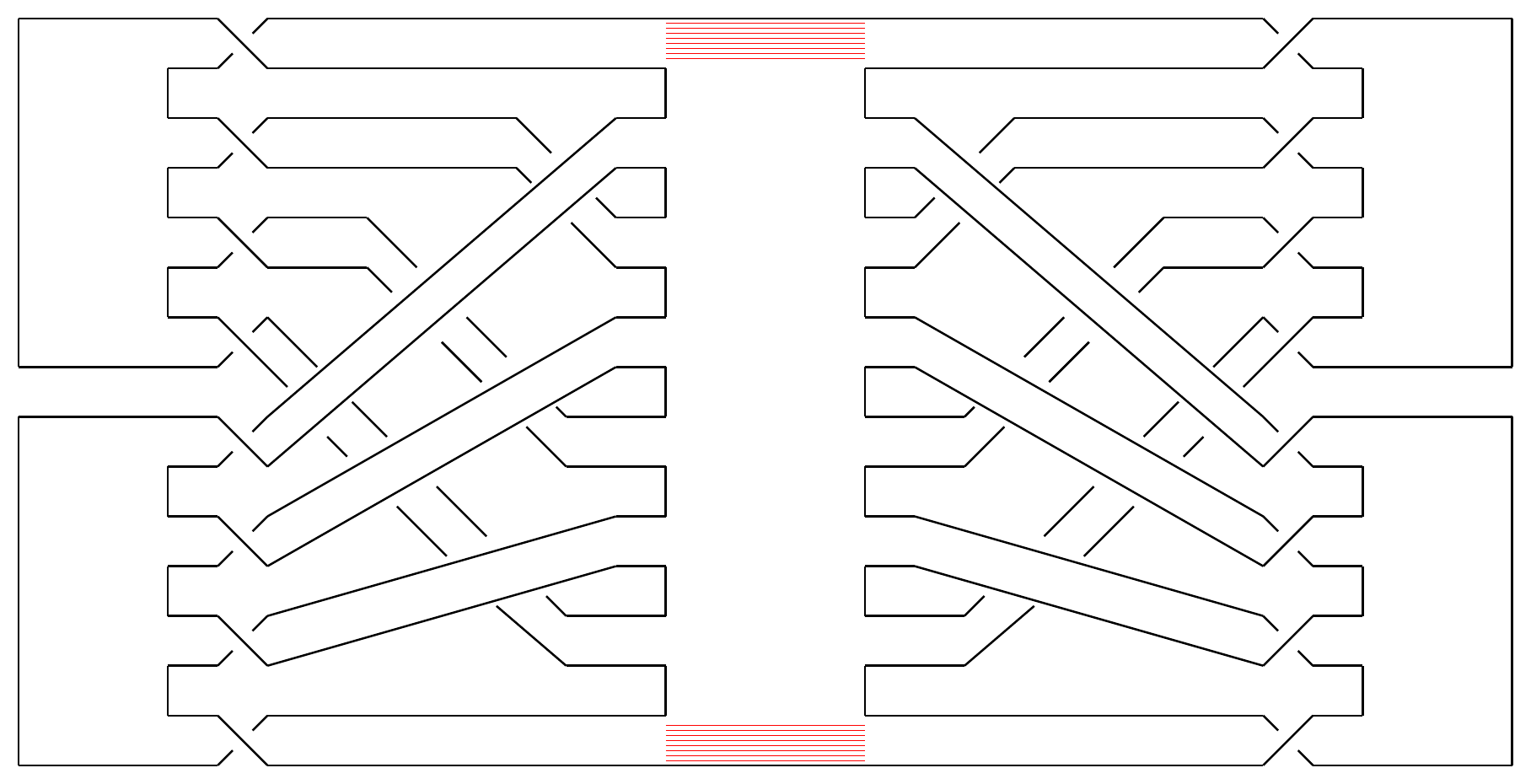}
\caption{For $Q_{4,3}$, connecting the ends of the red arcs in the natural way to produce the a link in $F_{p,q}$}
\label{fig: connecting ends}  
\end{figure}

\section{Verification of equivalence}

Given any $p,q$, by traversing all possible $((a,b,s),t,v)$, the algorithm in the previous section gives a family of 2-component links $\{Q_{p,q}\cup V_{a,b,s,t,v}\}$. More generally, a slight variants of the above algorithm can also process multiple input of $((a,b,s),s,t)$ to obtain links with more than two components. In this work, 3-component links  of the form $Q_{p,q} \cup V_{a_1,b_1,s_1,t_1,v_1} \cup V_{a_2,b_2, s_2, t_2, v_2}$ are particularly helpful to obtain stable equivalence. 

The strategy for verifying equivalence is the following proposition:

\begin{proposition}
Suppose a 3-component link $L=Q\cup V_1\cup V_2$ is a 3R link and $L'$ is a 2-component sublink of $L$ which is a 2R link. Then $L'$ is stably equivalent to $L$.
\end{proposition}

\begin{proof}
Let $K$ be the knot in $L$ which is not a component of $L'$. By Theorem 5.1 of \cite{Classification}, $K$ is isotopic in $S^3_{L'}(0) = \#^2 S^1\times S^2$ to an unknot contained in a 3-ball. Since isotopy of $K$ in $S^3_{L'}(0)$ can be realized as handleslides of $K$ over $L'$, it follows that $L$ is handleslide equivalent to $L' \sqcup \text{unknot}$. So $L$ is stably equivalent to $L'$.

\end{proof}

We can use Proposition 3.1 to verify Theorem1.1.

By using the verified computation, isometry signature, of snapPy \cite{snappy} inside SageMath \cite{sage} with our construction of links of the form $Q_{p,q} \cup V_1 \cup V_2$, we find that for each $1\le d\le 39$, there is some 2-component sublink of $L(3,2;\frac{4}{d})$ isotopic to some 2-component sublink of $L(3,2;\frac{4}{d+2})$.
So $L(3,2;\frac{4}{d})$ is stably equivalent to $L(3,2;\frac{4}{d+2})$ for $d\le 39$. Similarly, $L(3,2;\frac{4}{5})$ and $L(3,2;\frac{2}{3})$ also share an isotopic link and  $L(3,2;\frac{2}{3})$ is handleslide trivial, it follows that each $L(3,2;\frac{4}{d})$ is stably handleslide trivial for $d\le 39$.

Similarly, $L(3,2; \frac{6}{6n\pm 1})$ is stably handleslide equivalent to $L(3n\pm 1, 3; \frac{2}{3})$ for $n\le 5$ and GST link $L_n$ is stably handleslide equivalent to $L(n+1,n; \frac{2}{3})$ for $n\le 36$.

\end{document}